\documentclass[11pt]{amsart}

\usepackage{amsmath,amssymb,amsthm,mathtools}
\usepackage{graphicx}
\usepackage[colorlinks=true,linkcolor=blue,citecolor=blue,urlcolor=blue]{hyperref}
\usepackage{tikz}
\usetikzlibrary{arrows.meta,calc,positioning}

\newtheorem{theorem}{Theorem}[section]
\newtheorem{proposition}[theorem]{Proposition}
\newtheorem{lemma}[theorem]{Lemma}
\newtheorem{corollary}[theorem]{Corollary}
\theoremstyle{definition}
\newtheorem{definition}[theorem]{Definition}
\newtheorem{remark}[theorem]{Remark}
\newtheorem{example}[theorem]{Example}
\newtheorem{question}[theorem]{Question}
\newtheorem{conjecture}[theorem]{Conjecture}
\newtheorem{assumption}[theorem]{Assumption}

\newcommand{\RR}{\mathbb{R}}
\newcommand{\len}{\operatorname{Len}}
\newcommand{\Thi}{\operatorname{Thi}}
\newcommand{\Rop}{\operatorname{Rop}}
\newcommand{\Cr}{\operatorname{Cr}}
\newcommand{\diam}{\operatorname{diam}}

\newcommand{\CRad}{\operatorname{CRad}}
\newcommand{\Pack}{\operatorname{Pack}}

\newcommand{\Conv}{\operatorname{Conv}}
\newcommand{\Vol}{\operatorname{Vol}}
\newcommand{\trunk}{\operatorname{trunk}}
\newcommand{\strunk}{\operatorname{strunk}}

\newcommand{\deltaK}{\operatorname{distort}}
\newcommand{\Hg}{\mathcal{H}}
\newcommand{\Pol}{\mathcal{P}}
\newcommand{\eps}{\varepsilon}
\newcommand{\sD}{\mathsf{D}}

\title[Geometric densities and compression radii]{Geometric densities and compression radii of knot types}

\author{Makoto Ozawa}
\address{Department of Natural Sciences, Komazawa University, Tokyo, Japan}
\email{w3c@komazawa-u.ac.jp}


\begin{document}

\begin{abstract}
We introduce scale-free compression radii and packing ratios of knot types and clarify their relation to geometric densities.  Given a Euclidean-invariant, scale-covariant size functional \(D\) on embedded curves, we define the \(D\)-density by \(\len(\gamma)/D(\gamma)\) and the \(D\)-compression radius by \(D(\gamma)/\Thi(\gamma)\).  At the level of a single representative these quantities give the elementary factorization
\[
 \frac{\len(\gamma)}{\Thi(\gamma)}
 =
 \frac{\len(\gamma)}{D(\gamma)}\cdot
 \frac{D(\gamma)}{\Thi(\gamma)}.
\]
The point of the paper is not this cancellation itself, but the separation it suggests after optimization: density, compression, and ropelength generally have different minimizing sequences.  We give a criterion for equality after optimization, record the basic optimized inequality, and compute the unknot case for diameter and minimal enclosing radius.  We then prove polygonal approximation theorems for compression radii for \(D=\diam\) and \(D=R_{\min}\), using standard convergence of polygonal thickness, and formulate the corresponding hypotheses for other \(L^p\) size functionals.  Finally, we discuss how the resulting quantities interact with distortion, trunk, and supertrunk.  The framework is intended as a structural companion to density-type invariants rather than as an immediate improvement of the strongest known ropelength lower bounds.  In particular, the formal decomposition alone does not produce new ropelength lower bounds; such bounds would require independent estimates for the density and compression factors.
\end{abstract}

\maketitle

%==================================================
\section{Introduction}
%==================================================

Ropelength is one of the central scale-invariant quantities in geometric knot theory.  If $\gamma\subset\RR^3$ is a rectifiable embedded closed curve with positive thickness, its ropelength is
\[
 \Rop(\gamma)=\frac{\len(\gamma)}{\Thi(\gamma)}.
\]
For a knot type $K$, the invariant $\Rop(K)$ is obtained by minimizing this ratio over all representatives of $K$.  In this form, ropelength measures how much length is required relative to the thickness of the tube that can be placed around the curve.  The existence and structure of ropelength minimizers have been studied in depth; see, for example, \cite{CKS,Diao,RawdonSimon06}.

The purpose of the present paper is to separate the ropelength ratio into two geometrically different factors.  Let $D(\gamma)$ be a size functional measuring the spatial spread of $\gamma$.  Examples include diameter, minimal enclosing radius, radius of gyration, $L^p$-radial size, pairwise $L^p$-spread, and the cube root of the volume of the convex hull.  Then each representative satisfies the elementary but useful identity
\[
 \frac{\len(\gamma)}{\Thi(\gamma)}
 =
 \frac{\len(\gamma)}{D(\gamma)}
 \cdot
 \frac{D(\gamma)}{\Thi(\gamma)}.
\]
The first factor is a density-type quantity: it measures how much length is distributed relative to the chosen spatial scale $D$.  The second factor is a compression-radius-type quantity: it measures how large the spatial scale is relative to the thickness.  Thus ropelength can be viewed as
\[
 \text{ropelength}=\text{geometric density}\times\text{scale-free compression radius}.
\]
This is the guiding principle of the paper.

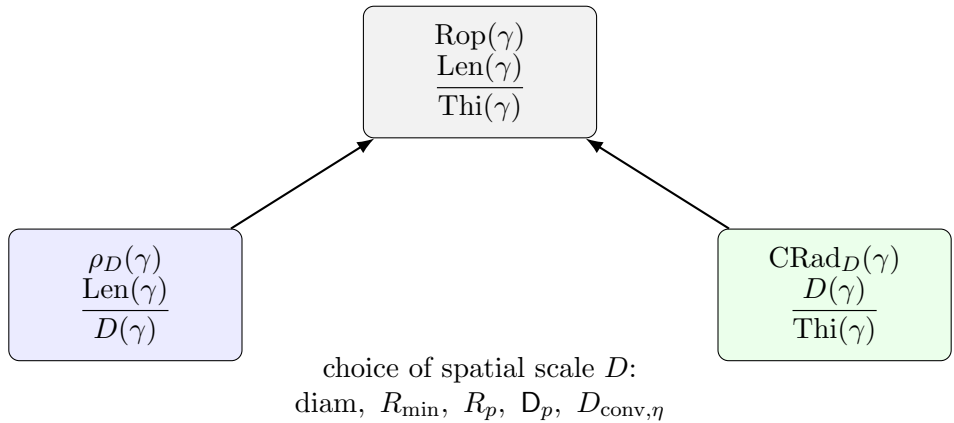
\begin{figure}[t]
\centering
\begin{tikzpicture}[>=Latex, node distance=12mm and 16mm, box/.style={draw, rounded corners, align=center, inner sep=6pt, minimum width=31mm}]
\node[box, fill=gray!10] (rop) {$\Rop(\gamma)$\\ $\displaystyle \frac{\len(\gamma)}{\Thi(\gamma)}$};
\node[box, fill=blue!8, below left=of rop] (dens) {$\rho_D(\gamma)$\\ $\displaystyle \frac{\len(\gamma)}{D(\gamma)}$};
\node[box, fill=green!8, below right=of rop] (comp) {$\CRad_D(\gamma)$\\ $\displaystyle \frac{D(\gamma)}{\Thi(\gamma)}$};
\draw[->, thick] (dens) -- (rop);
\draw[->, thick] (comp) -- (rop);
\node[align=center] at ($(dens)!0.5!(comp)+(0,-1.25)$) {choice of spatial scale $D$:\\ $\diam,\ R_{\min},\ R_p,\ \mathsf{D}_p,\ D_{\mathrm{conv},\eta}$};
\end{tikzpicture}
\caption{The representative-level decomposition of ropelength into a density factor and a compression factor. After optimization over a knot type, the minimizing representatives for the three quantities may differ.}
\label{fig:decomposition-schematic}
\end{figure}

Ratios of length to a spatial scale, especially those involving diameter or radius of gyration, are not new in geometric knot theory and polymer geometry.  Likewise, thickness and tube packing are central in the ropelength literature.  The contribution of the present paper is more modest and more specific: we organize these familiar quantities into a two-factor framework and introduce the complementary optimized invariant $D(\gamma)/\Thi(\gamma)$ as a scale-free compression radius.  The subsequent questions concern the failure of simultaneous optimization and the behavior of the compression factor under polygonal approximation.

The framework is meant to be read against the large literature on ropelength bounds in terms of crossing number.  Let \(\Cr(K)\) denote the crossing number.  With our convention \(\Thi\) is the tube radius; if one writes the tube diameter as \(\tau\), then \(\Thi=\tau/2\) and \(\Rop=\len/\Thi=\len/(\tau/2)\).  Known results show that ropelength has several different asymptotic regimes.  For arbitrary knot types there are universal lower bounds of order \(\Cr(K)^{3/4}\), while Diao--Ernst--Por--Ziegler proved a quasilinear upper bound \(\Rop(K)\leq a\Cr(K)\log^5\Cr(K)\) for a universal constant \(a>0\) \cite{BuckSimon99,DiaoErnstPorZiegler}.  Denne--Diao--Sullivan proved the quadrisecant lower bound \(\Rop(K)\geq 15.66\) for every nontrivial knot \cite{DenneDiaoSullivan}.  For alternating knots, Diao proved the linear lower bound \(\Rop(K)>b_0\Cr(K)\) with \(b_0>1/56\) \cite{DiaoAlt}.  By contrast, non-alternating torus knots exhibit the classical \(3/4\)-power behavior: their ropelengths admit constructions of order \(\Cr^{3/4}\), and recent work using concentric helices substantially improves the constants in this regime \cite{KlotzThompson}.  These results provide benchmarks for any proposed density-compression decomposition: to recover a known lower bound through the present framework, one must prove corresponding independent lower bounds for both the density and compression factors, or for their product.

This factorization is formally simple.  Accordingly, we do not present it as a deep ropelength theorem.  The mathematical content developed below is instead organized around three questions: when can equality survive after optimization over a knot type, for which size functionals can the compression term be approximated polygonally, and how do the resulting compression terms compare with existing invariants such as distortion and trunk?  In this form, the framework is meant to isolate a missing complementary invariant to density: the ratio between spatial spread and thickness.

The distinction between density and compression is useful.  A density such as $\len(\gamma)/D(\gamma)$ records how much curve is present per unit of spatial spread.  By contrast, $D(\gamma)/\Thi(\gamma)$ records how difficult it is to place a tube of radius comparable to $\Thi(\gamma)$ inside a region of size $D(\gamma)$.  These are not the same geometric feature.  A curve may have small density because it is long but widely spread, while it may have large compression radius because its spread is large relative to its thickness.  Conversely, a highly packed curve may have large density but small enclosing scale.

After optimizing over a knot type, the exact product formula generally becomes only an inequality.  If
\[
 \rho_D(K)=\inf_{\gamma\in K}\frac{\len(\gamma)}{D(\gamma)},\qquad
 \CRad_D(K)=\inf_{\gamma\in K}\frac{D(\gamma)}{\Thi(\gamma)},
\]
then
\[
 \Rop(K)\geq \rho_D(K)\CRad_D(K).
\]
Equality would require, in a strong sense, a single representative that simultaneously optimizes ropelength, density, and compression radius.  Such simultaneous optimization is not expected in general.

A second goal is to connect the smooth and polygonal theories.  For a polygonal knot $P$, the same exact factorization holds:
\[
 \Rop(P)=\rho_D(P)\CRad_D(P).
\]
For a fixed number $n$ of edges, one obtains polygonal invariants by minimizing over embedded $n$-gons of type $K$.  Again, equality need not survive after optimization.  However, under natural convergence hypotheses for polygonal thickness and the size functional $D$, the polygonal compression radii converge to the continuous compression radius as $n\to\infty$.

The paper also records two connections to existing geometric invariants.  First, when $D$ is the minimal enclosing radius, the corresponding density gives a lower bound for distortion up to a universal constant.  Second, trunk and supertrunk may be interpreted as sectional shadows of density and compression.  This leads to natural questions asking whether large trunk or supertrunk forces lower bounds on compression radii or density.

Finally, the present framework should be read together with three companion
works.  In \cite{OzawaIdealStratum}, the ropelength-filtered spaces of thick
representatives are studied through their ideal strata, admissible components,
and merge scales.  In \cite{OzawaDiscrete}, that filtered picture is
discretized into lattice-filtered move graphs, in which components and merge
scales become finitely computable.  Verified seed-generated computations there
determine BFACF merge scales for mirror seed pairs of the amphichiral knots
$4_1$ and $6_3$ (values $32$ and $44$, from birth levels $30$ and $40$, so
mirror barriers $2$ and $4$).  Discrete counterparts of density and compression
radii are developed directly in \cite{OzawaDiscretePDensity}, on
length-filtered lattice representative sets and seed-generated finite search
spaces.  The continuous invariants introduced below provide a natural reference
point for comparing those finite profiles with sublevel filtrations in smooth
knot spaces; see Section~\ref{sec:questions}.

The paper is organized as follows.  Section~\ref{sec:size} introduces admissible size functionals.  Section~\ref{sec:density-compression} defines density, compression radius, and packing ratio, and proves the ropelength decomposition and its optimized inequality.  Section~\ref{sec:examples} discusses concrete choices of $D$.  Section~\ref{sec:lp} gives $L^p$-radial and pairwise $L^p$ versions.  Section~\ref{sec:polygonal} treats polygonal approximation.  Section~\ref{sec:distortion-trunk} discusses distortion, trunk, and supertrunk.  Section~\ref{sec:questions} lists questions and conjectures.

%==================================================
\section{Size functionals for embedded curves}\label{sec:size}
%==================================================

Let $\Hg(K)$ denote the class of admissible representatives of a knot type $K$.  For the smooth theory, one may take $\Hg(K)$ to be the set of $C^{1,1}$ embedded closed curves in $\RR^3$ representing $K$ and having positive thickness.  The $C^{1,1}$ condition is natural in ropelength theory because thickness controls curvature in this regularity class.
Unless otherwise stated, convergence of smooth representatives means convergence in \(C^1\) together with convergence of arclength parametrizations and the relevant thickness semicontinuity.  In the polygonal section, convergence is understood after arclength parametrization and, when needed, after subdivision so that the number of edges tends to infinity.

\begin{definition}[Geometric size functional]
A geometric size functional is a map $D$ assigning to each admissible embedded closed curve $\gamma\subset\RR^3$ a real number $D(\gamma)$, satisfying the following properties:
\begin{enumerate}
 \item $D(g\gamma)=D(\gamma)$ for every Euclidean isometry $g$ of $\RR^3$;
 \item $D(a\gamma)=aD(\gamma)$ for every scale factor $a>0$;
 \item $D(\gamma)>0$ for every admissible closed curve $\gamma$.
\end{enumerate}
If, in addition, $D$ behaves continuously under the chosen topology on curves, we call $D$ a continuous geometric size functional.
\end{definition}

The scale-covariance condition is the key point.  It ensures that ratios such as $\len(\gamma)/D(\gamma)$ and $D(\gamma)/\Thi(\gamma)$ are scale-invariant.

\begin{example}[Diameter]
The diameter
\[
 \diam(\gamma)=\sup_{x,y\in\gamma}|x-y|
\]
is a geometric size functional.
\end{example}

\begin{example}[Minimal enclosing radius]
The minimal enclosing radius is
\[
 R_{\min}(\gamma)=\inf_{a\in\RR^3}\sup_{x\in\gamma}|x-a|.
\]
Equivalently, $R_{\min}(\gamma)$ is the radius of the smallest closed Euclidean ball containing $\gamma$.  It satisfies
\[
 \frac12\diam(\gamma)\leq R_{\min}(\gamma)\leq \diam(\gamma).
\]
\end{example}

\begin{example}[Convex-hull size]
If \(\Conv(\gamma)\) denotes the convex hull of \(\gamma\), then
\[
 D_{\mathrm{conv}}(\gamma)=\Vol(\Conv(\gamma))^{1/3}
\]
is scale-covariant whenever the convex hull has positive volume.  This functional is not positive on all admissible curves: a planar representative has convex-hull volume zero.  A safe regularized version is
\[
 D_{\mathrm{conv},\eta}(\gamma)=
 \left(\Vol(\Conv(\gamma))+\eta\diam(\gamma)^3\right)^{1/3},
 \qquad \eta>0.
\]
It is Euclidean-invariant, scale-covariant, and positive, since every embedded closed curve has positive diameter and hence
\[
 D_{\mathrm{conv},\eta}(\gamma)\geq \eta^{1/3}\diam(\gamma)>0.
\]
\end{example}

\begin{example}[Radius of gyration]
Let $\bar x=(1/\len(\gamma))\int_\gamma x\,ds$ be the arclength barycenter.  The radius of gyration is
\[
 R_g(\gamma)=\left(\frac{1}{\len(\gamma)}\int_\gamma |x-\bar x|^2\,ds\right)^{1/2}.
\]
This is the $L^2$-radial size of the curve and is widely used in polymer geometry.
\end{example}

%==================================================
\section{Density, compression radius, and packing ratio}\label{sec:density-compression}
%==================================================

Fix a geometric size functional $D$.

\begin{definition}[$D$-density]
For an admissible representative $\gamma$, define the $D$-density by
\[
 \rho_D(\gamma)=\frac{\len(\gamma)}{D(\gamma)}.
\]
For a knot type $K$, define
\[
 \rho_D(K)=\inf_{\gamma\in\Hg(K)}\rho_D(\gamma)
 =\inf_{\gamma\in\Hg(K)}\frac{\len(\gamma)}{D(\gamma)}.
\]
\end{definition}

\begin{definition}[Scale-free compression radius]
For an admissible representative $\gamma$ of positive thickness, define the $D$-compression radius by
\[
 \CRad_D(\gamma)=\frac{D(\gamma)}{\Thi(\gamma)}.
\]
For a knot type $K$, define
\[
 \CRad_D(K)=\inf_{\gamma\in\Hg(K)}\CRad_D(\gamma)
 =\inf_{\gamma\in\Hg(K)}\frac{D(\gamma)}{\Thi(\gamma)}.
\]
\end{definition}

\begin{definition}[Packing ratio]
The $D$-packing ratio of a representative is
\[
 \Pack_D(\gamma)=\frac{\Thi(\gamma)}{D(\gamma)}.
\]
The $D$-packing ratio of a knot type is
\[
 \Pack_D(K)=\sup_{\gamma\in\Hg(K)}\Pack_D(\gamma).
\]
Whenever $\CRad_D(K)>0$, one has
\[
 \Pack_D(K)=\frac{1}{\CRad_D(K)}.
\]
\end{definition}

\begin{remark}
The terminology is chosen so that smaller $\CRad_D$ means better compression, while larger $\Pack_D$ means better packing.  These are inverse viewpoints on the same scale-free phenomenon.
\end{remark}

\begin{lemma}[Representative-level ropelength decomposition]\label{prop:representative-decomposition}
For every admissible representative \(\gamma\) of positive thickness,
\[
 \Rop(\gamma)=\rho_D(\gamma)\CRad_D(\gamma).
\]
Equivalently,
\[
 \frac{\len(\gamma)}{\Thi(\gamma)}
 =
 \frac{\len(\gamma)}{D(\gamma)}\cdot
 \frac{D(\gamma)}{\Thi(\gamma)}.
\]
\end{lemma}

\begin{proof}
This is immediate by cancellation of \(D(\gamma)\).
\end{proof}

\begin{lemma}[Optimized ropelength inequality]\label{lem:optimized-inequality}
For every knot type \(K\) and every geometric size functional \(D\),
\[
 \Rop(K)\geq \rho_D(K)\CRad_D(K).
\]
\end{lemma}

\begin{proof}
For every \(\gamma\in\Hg(K)\), Lemma~\ref{prop:representative-decomposition} gives
\[
 \Rop(\gamma)=\rho_D(\gamma)\CRad_D(\gamma)
 \geq \rho_D(K)\CRad_D(K).
\]
Taking the infimum over \(\gamma\in\Hg(K)\) gives the result.
\end{proof}

\begin{proposition}[Equality criterion for optimized decomposition]\label{prop:equality-criterion}
Assume that \(\rho_D(K)>0\) and \(\CRad_D(K)>0\).  Then
\[
 \Rop(K)=\rho_D(K)\CRad_D(K)
\]
holds if and only if there exists a sequence \(\gamma_j\in\Hg(K)\) such that
\[
 \Rop(\gamma_j)\to\Rop(K),\qquad
 \rho_D(\gamma_j)\to\rho_D(K),\qquad
 \CRad_D(\gamma_j)\to\CRad_D(K).
\]
In particular, if all three infima are attained, equality holds if and only if a common minimizer exists.
\end{proposition}

\begin{proof}
If such a sequence exists, then
\[
 \Rop(K)\leq\lim_j\Rop(\gamma_j)
 =
 \lim_j\rho_D(\gamma_j)\CRad_D(\gamma_j)
 =
 \rho_D(K)\CRad_D(K),
\]
and the reverse inequality is Lemma~\ref{lem:optimized-inequality}.  Conversely, suppose equality holds and choose a minimizing sequence for ropelength, so that \(\Rop(\gamma_j)\to\Rop(K)\).  By the representative-level decomposition,
\[
 \rho_D(\gamma_j)\CRad_D(\gamma_j)=\Rop(\gamma_j)\to\Rop(K)=\rho_D(K)\CRad_D(K).
\]
In particular, the products are bounded.  Since
\[
 \rho_D(\gamma_j)\geq\rho_D(K)>0,\qquad
 \CRad_D(\gamma_j)\geq\CRad_D(K)>0,
\]
this boundedness rules out divergence of either factor.  The convergence of the product to the product of the two positive lower bounds then implies that both factors tend to their respective infima.
\end{proof}

\begin{remark}[Why equality is not automatic]
The equality
\[
 \Rop(K)=\rho_D(K)\CRad_D(K)
\]
requires the minimization of a product to agree with the product of two independently minimized factors.  Proposition~\ref{prop:equality-criterion} shows that this is equivalent, up to minimizing sequences, to simultaneous optimization of ropelength, density, and compression radius.
\end{remark}

\begin{example}[The unknot]
Let \(U\) be the unknot and let \(\gamma\) be a round circle of radius \(r\).  Then
\[
 \len(\gamma)=2\pi r,
 \qquad
 \Thi(\gamma)=r,
 \qquad
 R_{\min}(\gamma)=r,
 \qquad
 \diam(\gamma)=2r.
\]
For \(D=R_{\min}\) one obtains
\[
 \rho_R(U)=2\pi,
 \qquad
 \CRad_R(U)=1,
 \qquad
 \Rop(U)=2\pi,
\]
so equality holds.  For \(D=\diam\) one obtains
\[
 \rho_{\diam}(U)=\pi,
 \qquad
 \CRad_{\diam}(U)=2,
 \qquad
 \Rop(U)=2\pi,
\]
again with equality.  Thus the optimized product can be sharp, even though equality is not expected without a common minimizing sequence.
\end{example}

\begin{example}[Strictness for an artificial size functional]\label{ex:artificial-strictness}
The optimized inequality can be strict.  This can be seen without relying on any exact ropelength computation for a nontrivial knot.  Define
\[
 D_*(\gamma)=\Thi(\gamma)\left(1+e^{-\Rop(\gamma)}\right).
\]
This is Euclidean-invariant, scale-covariant, and positive, although it is not a natural spatial size functional.  For each representative,
\[
 \rho_{D_*}(\gamma)=\frac{\Rop(\gamma)}{1+e^{-\Rop(\gamma)}},\qquad
 \CRad_{D_*}(\gamma)=1+e^{-\Rop(\gamma)}.
\]
The function $x/(1+e^{-x})$ is strictly increasing for $x>0$, and every knot type admits representatives with arbitrarily large ropelength.  Indeed, starting from any smooth representative, one may insert inside a small ball a Reidemeister-I kink of scale \(\varepsilon\); this preserves the ambient isotopy type, while the thickness is at most of order \(\varepsilon\) and the length remains bounded below, so the ropelength tends to infinity as \(\varepsilon\to0\).  Hence
\[
 \rho_{D_*}(K)=\frac{\Rop(K)}{1+e^{-\Rop(K)}},
 \qquad
 \CRad_{D_*}(K)=1.
\]
Therefore
\[
 \rho_{D_*}(K)\CRad_{D_*}(K)
 =
 \frac{\Rop(K)}{1+e^{-\Rop(K)}}
 <
 \Rop(K).
\]
This example is included only to show that equality after optimization is a genuine condition, not a formal consequence of the representative-level identity.  For natural choices such as $D=\diam$ or $D=R_{\min}$, proving strictness for a specific nontrivial knot would require detailed information about the corresponding density and compression minimizers, which is not currently available.
\end{example}

\begin{remark}[Usefulness and limitation of the basic inequality]
The inequality
\[
 \Rop(K)\geq \rho_D(K)\CRad_D(K)
\]
becomes a useful lower bound only when both factors admit independent effective lower bounds.  For example, for \(D=\diam\) every closed curve satisfies \(\len(\gamma)\geq 2\diam(\gamma)\), hence \(\rho_{\diam}(K)\geq 2\) and
\[
 \Rop(K)\geq 2\CRad_{\diam}(K).
\]
This universal estimate is not intended to improve the best known ropelength bounds for standard knots.  Rather, it shows where new information would have to enter: one needs either a sharp density lower bound or a sharp compression lower bound.  The later questions on trunk and supertrunk are motivated by this need.
\end{remark}

\begin{definition}[Optimal strata associated with $D$]
When the relevant infima are attained, define
\[
 \mathcal I_{\rho_D}(K)=\left\{\gamma\in\Hg(K)\mid \rho_D(\gamma)=\rho_D(K)\right\},
\]
\[
 \mathcal I_{\CRad_D}(K)=\left\{\gamma\in\Hg(K)\mid \CRad_D(\gamma)=\CRad_D(K)\right\},
\]
and
\[
 \mathcal I_{\Rop}(K)=\left\{\gamma\in\Hg(K)\mid \Rop(\gamma)=\Rop(K)\right\}.
\]
We call these the density ideal stratum, compression ideal stratum, and ropelength ideal stratum, respectively.
\end{definition}

\begin{remark}
The compression ideal stratum should not be identified with the usual ideal knot stratum unless the corresponding minimizers coincide.  Thus a knot type may possess several different ideal strata, each associated with a different geometric functional.  The ropelength ideal stratum $\mathcal I_{\Rop}(K)$ is the initial layer of the ropelength filtration studied in \cite{OzawaIdealStratum}.  In the same way, each of the functionals $\rho_D$ and $\CRad_D$ defines its own sublevel filtration of $\Hg(K)$, with its own initial layer, admissible components, and merge scales; these filtered variants are taken up in Section~\ref{sec:questions}.
\end{remark}

%==================================================
\section{Basic examples and comparisons}\label{sec:examples}
%==================================================

\subsection{Diameter and minimal enclosing radius}

Let $D=\diam$ or $D=R_{\min}$.  Since
\[
 \frac12\diam(\gamma)\leq R_{\min}(\gamma)\leq \diam(\gamma),
\]
the corresponding densities and compression radii are equivalent up to universal constants.  More precisely,
\[
 \frac{\len(\gamma)}{\diam(\gamma)}
 \leq
 \frac{\len(\gamma)}{R_{\min}(\gamma)}
 \leq
 2\frac{\len(\gamma)}{\diam(\gamma)},
\]
and
\[
 \frac12\frac{\diam(\gamma)}{\Thi(\gamma)}
 \leq
 \frac{R_{\min}(\gamma)}{\Thi(\gamma)}
 \leq
 \frac{\diam(\gamma)}{\Thi(\gamma)}.
\]
Thus the diameter theory and the minimal-enclosing-ball theory measure the same coarse phenomenon, but the latter is better adapted to the language of compression and storage.

\begin{figure}[t]
\centering
\begin{tikzpicture}[scale=1.0]
\coordinate (c) at (0,0);
\draw[dashed, red!70] (c) circle (2.05);
\fill (c) circle (1.2pt) node[below right] {$a$};
\draw[thick, blue] plot[smooth cycle, tension=.88] coordinates {(-1.55,-0.35) (-1.10,1.15) (0.05,1.45) (1.20,0.70) (1.35,-0.85) (0.10,-1.45) (-1.05,-1.00)};
\draw[thick] (-1.55,-0.35) -- (1.20,0.70);
\node[above] at (-0.15,0.35) {$\diam(\gamma)$};
\draw[->, thick, red!70] (0,0) -- (1.35,1.40);
\node[anchor=west, red!70!black] at (1.60,1.62) {$R_{\min}(\gamma)$};
\end{tikzpicture}
\caption{Two basic size functionals. The dashed circle represents a minimal enclosing ball of radius $R_{\min}(\gamma)$ centered at $a$, while the segment indicates the diameter scale $\diam(\gamma)$.}
\label{fig:diameter-enclosing-radius}
\end{figure}
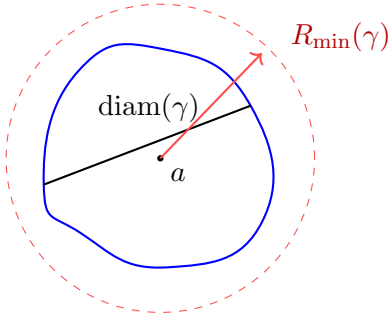

\subsection{Convex-hull compression}

The convex-hull size $D_{\mathrm{conv}}(\gamma)=\Vol(\Conv(\gamma))^{1/3}$ is sensitive to the actual three-dimensional occupation of the curve.  It may distinguish a nearly planar long curve from a curve that fills a genuinely three-dimensional region.  Degeneracy of the convex hull must be controlled.  One possible regularized choice is
\[
 D_{\mathrm{conv},\eta}(\gamma)=
 \left(\Vol(\Conv(\gamma))+\eta\diam(\gamma)^3\right)^{1/3},
 \qquad \eta>0.
\]
This regularized size is positive and scale-covariant.

\subsection{Interpretation of compression}

If $D=R_{\min}$, the quantity
\[
 \CRad_R(\gamma)=\frac{R_{\min}(\gamma)}{\Thi(\gamma)}
\]
measures the radius of the smallest enclosing ball in units of tube thickness.  The inverse
\[
 \Pack_R(\gamma)=\frac{\Thi(\gamma)}{R_{\min}(\gamma)}
\]
measures how large the tube thickness is relative to the enclosing radius.  Thus a large packing ratio means that the representative is efficiently stored in a small ball relative to its tube radius.

%==================================================
\section{Ropelength benchmarks and consequences for the decomposition}\label{sec:ropelength-benchmarks}
%==================================================

We record several known ropelength estimates in order to clarify what the density-compression framework can and cannot presently prove.  Some of the most recent results cited below, notably the works of Klotz--Thompson and Klotz on torus links, are currently available as arXiv preprints; they are used here as benchmarks and motivation rather than as ingredients in the proofs.  Buck--Simon is cited for the foundational relationship between thickness and crossing number, while Diao--Ernst--Por--Ziegler provides the quasilinear upper-bound side of the general ropelength picture.  The statements in this section are not new; they are included to position the proposed invariants relative to established ropelength theory.

\subsection{General lower and upper bounds}

For arbitrary nontrivial knots, Denne--Diao--Sullivan used quadrisecants to prove
\[
 \Rop(K)\geq 15.66.
\]
This improved earlier universal lower bounds and is close to numerical upper estimates for the trefoil; see also the computational tightening work of Ashton--Cantarella--Piatek--Rawdon \cite{AshtonCantarellaPiatekRawdon}.  In terms of crossing number, there is a universal lower bound of the form
\[
 \Rop(K)\geq \alpha_0\Cr(K)^{3/4}
\]
for a positive constant \(\alpha_0\); the currently established universal constants are far from the values suggested by numerical and constructive examples.  On the other hand, Diao--Ernst--Por--Ziegler proved that there exists a universal constant \(a>0\) such that
\[
 \Rop(K)\leq a\Cr(K)\log^5\Cr(K)
\]
for every knot or link \(K\).  Thus ropelength is known to be at most quasilinear in crossing number.

In the language of the present paper, any proof of a ropelength lower bound through a size functional \(D\) would have to establish
\[
 \rho_D(K)\CRad_D(K)\geq F(K)
\]
for the desired lower-bound function \(F(K)\).  The optimized inequality then gives \(\Rop(K)\geq F(K)\).  Conversely, an existing lower bound for \(\Rop(K)\) alone does not imply a lower bound for \(\rho_D(K)\CRad_D(K)\), since the optimized product may be strictly smaller than ropelength.

The universal estimates can be translated into targets for the present decomposition.  For instance, if $D=\diam$, then every closed curve satisfies $\len(\gamma)\geq 2\diam(\gamma)$, so $\rho_{\diam}(K)\geq2$.  Thus any lower bound of the form
\[
 \CRad_{\diam}(K)\geq c\Cr(K)^{3/4}
\]
would imply a ropelength lower bound of the same order, with constant $2c$.  Similarly, an alternating-knot lower bound would follow from a compression estimate of order $\Cr(K)$, or from a mixed estimate in which both the density and the compression factors grow.  The point is not that such estimates are known here, but that known ropelength bounds specify what one should try to prove for the two factors separately.

\subsection{Alternating knots}

For alternating knots, Diao proved the long-standing conjecture that ropelength is bounded below linearly by crossing number.  More precisely, there is a constant \(b_0>1/56\) such that
\[
 \Rop(K)>b_0\Cr(K)
\]
for every alternating knot \(K\).  Therefore, if one hopes to recover the alternating-knot lower bound by the present decomposition, one needs a mechanism forcing
\[
 \rho_D(K)\CRad_D(K)\gtrsim \Cr(K)
\]
for alternating knots.  This suggests that alternating knots should have either large density, large compression radius, or both, for suitable choices of \(D\).  Thus the alternating ropelength phenomenon may be reformulated as a question about how complexity is distributed between length density and thickness-normalized spatial spread.

For the alternating torus knots and links of type \(T(2,n)\), upper-bound constructions are closely related to double-helical and superhelical patterns.  The classical double-helix constructions give linear upper bounds in the crossing number.  Later constructions using superhelices and non-identical double helices have reduced the coefficient \cite{HuhKimOh2018,KlotzMaldonado2021,KimOhHuh2024}.  These constructions are useful test cases for the invariants \(\rho_D\) and \(\CRad_D\), because their geometry is explicit enough that one can try to estimate the two factors separately.

\subsection{Non-alternating torus knots}

Non-alternating torus knots display a different asymptotic behavior.  The ropelength of many such families grows like the \(3/4\)-power of crossing number rather than linearly.  This behavior was already visible in earlier lattice and helical constructions, and recent work of Klotz--Thompson using concentric helices gives efficient configurations for non-alternating torus knots, with an optimized length on the order of \(Q^{3/2}\) for a family built from \(Q\) helices.  Since the crossing number in these families is quadratic in the relevant parameter, this corresponds to the \(\Cr^{3/4}\)-scale.

This distinction is important for the present paper.  If a single size functional \(D\) is used for both alternating and non-alternating torus families, then the product
\[
 \rho_D(K)\CRad_D(K)
\]
must be flexible enough to reflect both the linear alternating regime and the \(3/4\)-power non-alternating regime.  Thus torus knots provide a natural laboratory for testing whether density or compression is the dominant contributor to ropelength growth.

\subsection{Torus links and close packing}

For torus links \(T(Q,Q)\), Klotz recently obtained strong lower bounds using the convex hull of close-packed unit disks \cite{KlotzTQQ}.  In particular, the coefficient in the \(\Cr^{3/4}\)-type lower bound is much larger for these torus links than the best currently known universal coefficient.  This result is especially relevant to the present framework because it uses a packing and convex-hull mechanism.  It suggests that convex-hull size, enclosing radius, and compression radius are not merely auxiliary quantities, but can encode part of the geometry responsible for sharp ropelength lower bounds.

\begin{remark}[Role of the present invariants]
The preceding results show that the present paper should not claim to improve existing ropelength bounds merely from the formal product decomposition.  Rather, the contribution is to isolate two factors whose separate behavior can be studied.  The known ropelength estimates give target scales.  For example, one may ask whether for alternating knots some natural \(D\) satisfies
\[
 \rho_D(K)\CRad_D(K)\geq c\Cr(K),
\]
whereas for non-alternating torus knots one may ask whether the same or a different \(D\) yields the scale
\[
 \rho_D(K)\CRad_D(K)\asymp \Cr(K)^{3/4}.
\]
These are not consequences of the present formalism; they are concrete problems suggested by it.
\end{remark}

%==================================================
\section{$L^p$-radial compression and pairwise $p$-densities}\label{sec:lp}
%==================================================

The preceding definitions apply to any scale-covariant size functional.  Two natural $L^p$ families are particularly useful: radial sizes and pairwise spreads.

\subsection{$L^p$-radial size}

\begin{definition}[$L^p$-radial size]
For $1\leq p<\infty$, define
\[
 R_p(\gamma)=
 \inf_{a\in\RR^3}
 \left(
 \frac{1}{\len(\gamma)}\int_\gamma |x-a|^p\,ds
 \right)^{1/p}.
\]
For $p=\infty$, define
\[
 R_\infty(\gamma)=\inf_{a\in\RR^3}\sup_{x\in\gamma}|x-a|=R_{\min}(\gamma).
\]
\end{definition}

The associated quantities are
\[
 \rho^{\mathrm{rad}}_p(\gamma)=\frac{\len(\gamma)}{R_p(\gamma)},
 \qquad
 \CRad^{\mathrm{rad}}_p(\gamma)=\frac{R_p(\gamma)}{\Thi(\gamma)},
\]
and similarly for a knot type $K$ by taking infima.

\begin{proposition}[$L^p$ radial ropelength decomposition]
For every representative $\gamma$ of positive thickness and every $1\leq p\leq\infty$,
\[
 \Rop(\gamma)
 =
 \rho^{\mathrm{rad}}_p(\gamma)\CRad^{\mathrm{rad}}_p(\gamma).
\]
Consequently,
\[
 \Rop(K)
 \geq
 \rho^{\mathrm{rad}}_p(K)\CRad^{\mathrm{rad}}_p(K).
\]
\end{proposition}

\begin{proof}
This is Proposition~\ref{prop:representative-decomposition} applied to $D=R_p$.
\end{proof}

\begin{proposition}[The case $p=2$]
For $p=2$, the minimizing center in the definition of $R_2(\gamma)$ is the arclength barycenter $\bar x$, and
\[
 R_2(\gamma)=R_g(\gamma).
\]
\end{proposition}

\begin{proof}
Expanding $\int_\gamma |x-a|^2\,ds$ gives
\[
 \int_\gamma |x-a|^2\,ds
 =
 \int_\gamma |x-\bar x|^2\,ds
 +\len(\gamma)|a-\bar x|^2.
\]
Thus the minimum is attained at $a=\bar x$.
\end{proof}

\subsection{Pairwise $L^p$-spread}

For comparison with unified $p$-densities, one may use pairwise distance averages.

\begin{definition}[Pairwise $L^p$-spread]
For $p\in(0,\infty)$, define
\[
 \sD_p(\gamma)=
 \left(
 \frac{1}{\len(\gamma)^2}
 \int_\gamma\int_\gamma |x-y|^p\,ds_x\,ds_y
 \right)^{1/p}.
\]
For $p=\infty$, set $\sD_\infty(\gamma)=\diam(\gamma)$.
\end{definition}

Then one obtains pairwise density and pairwise compression radius by
\[
 \rho_p(\gamma)=\frac{\len(\gamma)}{\sD_p(\gamma)},
 \qquad
 \CRad_p(\gamma)=\frac{\sD_p(\gamma)}{\Thi(\gamma)}.
\]
Again
\[
 \Rop(\gamma)=\rho_p(\gamma)\CRad_p(\gamma).
\]

\begin{proposition}[Relation between $R_2$ and $\sD_2$]
For every admissible representative $\gamma$,
\[
 \sD_2(\gamma)=\sqrt{2}\,R_2(\gamma)=\sqrt{2}\,R_g(\gamma).
\]
\end{proposition}

\begin{proof}
Writing all integrals with respect to arclength, we compute
\begin{align*}
 &\frac{1}{\len(\gamma)^2}\int_\gamma\int_\gamma |x-y|^2\,ds_x\,ds_y \\
 &\quad=
 \frac{1}{\len(\gamma)^2}\int_\gamma\int_\gamma
 \bigl(|x-\bar x|^2+|y-\bar x|^2
 -2\langle x-\bar x,y-\bar x\rangle\bigr)
 \,ds_x\,ds_y.
\end{align*}
Since $\int_\gamma (x-\bar x)\,ds_x=0$, the cross term vanishes.  Hence
\[
 \frac{1}{\len(\gamma)^2}\int_\gamma\int_\gamma |x-y|^2\,ds_x\,ds_y
 =
 \frac{2}{\len(\gamma)}\int_\gamma |x-\bar x|^2\,ds_x.
\]
Taking square roots gives $\sD_2(\gamma)=\sqrt2 R_2(\gamma)$.
\end{proof}

%==================================================
\section{Polygonal theory and approximation}\label{sec:polygonal}
%==================================================

We now formulate polygonal analogues.  Let $\Pol_n(K)$ denote the set of embedded polygonal $n$-gons representing the knot type $K$.  A polygonal thickness $\Thi_{\mathrm{poly}}(P)$ may be defined using the usual minimum of local curvature radius at vertices and half the doubly critical self-distance, following the standard polygonal ropelength framework. Figure~\ref{fig:polygonal-approximation} summarizes the approximation picture proved later in this section.

\begin{figure}[t]
\centering
\begin{tikzpicture}[>=Latex, scale=0.64]
\begin{scope}[xshift=0cm]
\draw[thick, blue] plot[smooth cycle, tension=.9] coordinates {(-1.2,-0.3) (-0.8,1.0) (0.2,1.3) (1.2,0.5) (1.0,-0.8) (0.0,-1.2) (-1.0,-0.9)};
\node[font=\scriptsize] at (0,-1.75) {$\gamma$};
\end{scope}
\draw[->, thick] (1.80,0) -- (3.45,0);
\node[font=\scriptsize, above] at (2.62,0.92) {inscribed polygons};
\begin{scope}[xshift=5.65cm]
\draw[thick, blue!35] plot[smooth cycle, tension=.9] coordinates {(-1.2,-0.3) (-0.8,1.0) (0.2,1.3) (1.2,0.5) (1.0,-0.8) (0.0,-1.2) (-1.0,-0.9)};
\draw[very thick, orange] (-1.15,-0.30) -- (-0.75,0.95) -- (0.20,1.28) -- (1.15,0.50) -- (1.00,-0.75) -- (0.05,-1.18) -- (-1.00,-0.90) -- cycle;
\node[font=\scriptsize] at (0,-1.75) {$P_n$};
\end{scope}
\draw[->, thick] (7.55,0) -- (9.05,0);
\node[font=\scriptsize, above] at (8.15,0.72) {$n\to\infty$};
\begin{scope}[xshift=12.05cm]
\node[draw, rounded corners, fill=gray!10, align=center, inner sep=5pt, font=\scriptsize] {\(\displaystyle \CRad_{D,n}(K) \to \CRad_D(K)\)\\[2pt]\(\displaystyle \Pack_{D,n}(K) \to \Pack_D(K)\)};
\end{scope}
\end{tikzpicture}
\caption{Schematic picture of polygonal approximation: a smooth representative $\gamma$ is approximated by polygonal representatives $P_n$, and the polygonal compression and packing invariants converge to their smooth counterparts.}
\label{fig:polygonal-approximation}
\end{figure}
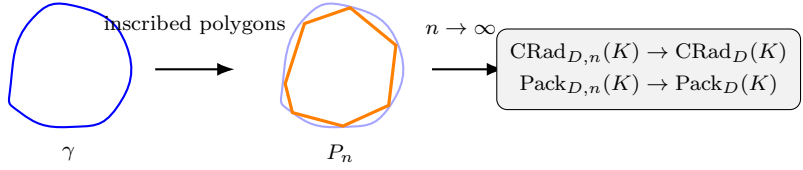

\begin{definition}[Polygonal invariants]
For $P\in\Pol_n(K)$, define
\[
 \rho_D(P)=\frac{\len(P)}{D(P)},
 \qquad
 \CRad_D(P)=\frac{D(P)}{\Thi_{\mathrm{poly}}(P)},
 \qquad
 \Pack_D(P)=\frac{\Thi_{\mathrm{poly}}(P)}{D(P)}.
\]
Define
\[
 \rho_{D,n}(K)=\inf_{P\in\Pol_n(K)}\rho_D(P),
\]
\[
 \CRad_{D,n}(K)=\inf_{P\in\Pol_n(K)}\CRad_D(P),
\]
and
\[
 \Pack_{D,n}(K)=\sup_{P\in\Pol_n(K)}\Pack_D(P).
\]
\end{definition}

\begin{proposition}[Polygonal representative-level decomposition]
For every polygonal representative $P$ with positive polygonal thickness,
\[
 \Rop_{\mathrm{poly}}(P)=\rho_D(P)\CRad_D(P).
\]
Consequently,
\[
 \Rop_{\mathrm{poly},n}(K)
 \geq
 \rho_{D,n}(K)\CRad_{D,n}(K).
\]
\end{proposition}

\begin{proof}
The first assertion follows by cancellation of $D(P)$.  The second follows by taking the infimum over $P\in\Pol_n(K)$.
\end{proof}

\begin{remark}
The equality
\[
 \Rop_{\mathrm{poly},n}(K)=\rho_{D,n}(K)\CRad_{D,n}(K)
\]
need not hold.  Even for fixed $n$, the polygon minimizing ropelength need not minimize density or compression radius.
\end{remark}

We state the approximation theorem in a form that isolates the analytic hypotheses needed for a particular choice of $D$.

\begin{assumption}[Approximation and compactness hypotheses]\label{ass:approx}
Let \(D\) be a geometric size functional.  Assume the following.
\begin{enumerate}
 \item For every \(C^{1,1}\) representative \(\gamma\in\Hg(K)\) of positive thickness, there exist polygons \(P_n\in\Pol_n(K)\) such that \(P_n\to\gamma\), \(D(P_n)\to D(\gamma)\), and \(\Thi_{\mathrm{poly}}(P_n)\to\Thi(\gamma)\).
 \item If \(P_n\in\Pol_n(K)\) satisfies \(\Thi_{\mathrm{poly}}(P_n)=1\) and \(D(P_n)\leq C\), then, after translations and passing to a subsequence, \(P_n\) converges to an admissible representative \(\gamma\in\Hg(K)\) of the same knot type.
 \item For every subsequence as in (2), the compression ratio is lower semicontinuous in the normalized form:
 \[
  \frac{D(\gamma)}{\Thi(\gamma)}
  \leq
  \liminf_{n\to\infty}
  \frac{D(P_n)}{\Thi_{\mathrm{poly}}(P_n)}.
 \]
\end{enumerate}
\end{assumption}

\begin{lemma}[Concrete verification for diameter and enclosing radius]\label{lem:diam-rmin-assumption}
For \(D=\diam\) and \(D=R_{\min}\), Assumption~\ref{ass:approx} holds in the standard polygonal ropelength setting.
\end{lemma}

\begin{proof}
For the approximation part, take sufficiently fine inscribed polygons in a \(C^{1,1}\) representative.  The convergence of polygonal thickness to smooth thickness is the main polygonal-thickness convergence theorem of Rawdon--Simon \cite{RawdonSimon06}.  Diameter is continuous under Hausdorff convergence.  The minimal enclosing radius is also continuous under Hausdorff convergence: if two compact sets have Hausdorff distance at most \(\eta\), then each is contained in the \(\eta\)-neighborhood of the other, and their minimal enclosing radii differ by at most \(\eta\).

For compactness, first consider \(D=\diam\).  After translation, the condition \(\diam(P_n)\leq C\) places all vertices and edges in a fixed ball.  The condition \(\Thi_{\mathrm{poly}}(P_n)=1\) gives an embedded unit tubular neighborhood.  Since this tube is contained in a fixed slightly larger ball, its volume gives a uniform bound for \(\len(P_n)\).  Arclength parametrizations therefore form an equi-Lipschitz family, and a subsequence converges uniformly.  The thickness lower bound prevents collapse to a self-intersecting or singular limit and gives \(\Thi(\gamma)\geq1\) by lower semicontinuity of thickness.  Knot type is preserved under sufficiently close convergence with a uniform embedded tube.  Moreover, in this normalized situation, \(\Thi(\gamma)\geq1\) and lower semicontinuity of diameter give
\[
 \frac{\diam(\gamma)}{\Thi(\gamma)}\leq \diam(\gamma)
 \leq \liminf_{n\to\infty}\diam(P_n)
 =\liminf_{n\to\infty}\frac{\diam(P_n)}{\Thi_{\mathrm{poly}}(P_n)}.
\]
Thus the required lower semicontinuity of the ratio holds for \(D=\diam\).

For \(D=R_{\min}\), the bound \(R_{\min}(P_n)\leq C\) implies, after translating the centers of minimal enclosing balls to a bounded region, that \(P_n\) is contained in a fixed ball of radius \(C+1\).  The same tube-volume and compactness argument applies.  Finally, \(R_{\min}\) is lower semicontinuous, indeed continuous under Hausdorff convergence, as observed above.  Since \(\Thi(\gamma)\geq1\) in the normalized limit,
\[
 \frac{R_{\min}(\gamma)}{\Thi(\gamma)}\leq R_{\min}(\gamma)
 \leq \liminf_{n\to\infty}R_{\min}(P_n)
 =\liminf_{n\to\infty}\frac{R_{\min}(P_n)}{\Thi_{\mathrm{poly}}(P_n)}.
\]
Hence the ratio lower semicontinuity also holds for \(D=R_{\min}\).
\end{proof}

\begin{theorem}[Polygonal approximation for compression radii]\label{thm:compression-approx}
Suppose that \(D\) satisfies Assumption~\ref{ass:approx}.  Then
\[
 \lim_{n\to\infty}\CRad_{D,n}(K)=\CRad_D(K).
\]
If \(\CRad_D(K)>0\), then
\[
 \lim_{n\to\infty}\Pack_{D,n}(K)=\Pack_D(K).
\]
In particular, the conclusion holds for \(D=\diam\) and \(D=R_{\min}\).
\end{theorem}

\begin{proof}
We prove the two inequalities.  For the upper limit, fix \(\eps>0\) and choose \(\gamma\in\Hg(K)\) such that
\[
 \frac{D(\gamma)}{\Thi(\gamma)}<\CRad_D(K)+\eps.
\]
By Assumption~\ref{ass:approx}, there are polygons \(P_n\in\Pol_n(K)\) with \(D(P_n)\to D(\gamma)\) and \(\Thi_{\mathrm{poly}}(P_n)\to\Thi(\gamma)\).  Therefore
\[
 \limsup_{n\to\infty}\CRad_{D,n}(K)
 \leq
 \lim_{n\to\infty}\frac{D(P_n)}{\Thi_{\mathrm{poly}}(P_n)}
 =
 \frac{D(\gamma)}{\Thi(\gamma)}
 <
 \CRad_D(K)+\eps.
\]
Letting \(\eps\to0\) gives
\[
 \limsup_{n\to\infty}\CRad_{D,n}(K)
 \leq \CRad_D(K).
\]

For the lower limit, let
\[
 L=\liminf_{n\to\infty}\CRad_{D,n}(K).
\]
If \(L=\infty\), there is nothing to prove.  Otherwise choose a subsequence, still denoted by \(n\), and polygons \(P_n\in\Pol_n(K)\) such that
\[
 \frac{D(P_n)}{\Thi_{\mathrm{poly}}(P_n)}
 \leq
 \CRad_{D,n}(K)+\eps_n,
 \qquad \eps_n\to0,
\]
and such that the right-hand side remains bounded.  By scale invariance, rescale so that \(\Thi_{\mathrm{poly}}(P_n)=1\).  Then \(D(P_n)\) is bounded along this subsequence.  Assumption~\ref{ass:approx} gives a further subsequence converging to some \(\gamma\in\Hg(K)\) with \(\Thi(\gamma)\geq1\) and
\[
 D(\gamma)\leq\liminf_{n\to\infty}D(P_n).
\]
By the ratio lower semicontinuity in Assumption~\ref{ass:approx},
\[
 \frac{D(\gamma)}{\Thi(\gamma)}
 \leq
 \liminf_{n\to\infty}\frac{D(P_n)}{\Thi_{\mathrm{poly}}(P_n)}
 =L.
\]
Therefore
\[
 \CRad_D(K)
 \leq
 \frac{D(\gamma)}{\Thi(\gamma)}
 \leq L.
\]
This proves
\[
 \CRad_D(K)\leq \liminf_{n\to\infty}\CRad_{D,n}(K).
\]
Combining the upper and lower inequalities proves the convergence of \(\CRad_{D,n}(K)\).  The statement for packing ratios follows by taking reciprocals.
\end{proof}

\begin{corollary}[Other size functionals]
The conclusion of Theorem~\ref{thm:compression-approx} applies to any choice of \(D\) for which Assumption~\ref{ass:approx} can be verified.  Natural further candidates include
\[
 D=R_p\quad(1\leq p<\infty),
 \qquad
 D=\sD_p\quad(1\leq p\leq\infty),
\]
and regularized convex-hull sizes \(D_{\mathrm{conv},\eta}\).  Their verification requires the corresponding continuity or semicontinuity statements for the chosen functional and is not automatic from the definitions alone.
\end{corollary}

\begin{remark}
The main extra difficulty for compression radii, compared with density approximation, is the presence of thickness.  Density approximation requires convergence of length and spatial size.  Compression approximation requires convergence or at least suitable semicontinuity of polygonal thickness.
\end{remark}

%==================================================
\section{Distortion, trunk, and supertrunk}\label{sec:distortion-trunk}
%==================================================

This section records relationships and questions connecting compression and density with other geometric invariants. Figure~\ref{fig:distortion-trunk-schematic} collects the two guiding geometric pictures.

\begin{figure}[t]
\centering
\begin{tikzpicture}[>=Latex, scale=0.95]
\begin{scope}[xshift=0cm]
\node at (0,2.15) {distortion};
\draw[thick, blue] plot[smooth cycle, tension=.9] coordinates {(-1.55,0.10) (-1.10,1.20) (0.20,1.00) (1.20,0.25) (1.00,-0.95) (-0.05,-1.25) (-1.25,-0.75)};
\fill (-0.18,0.98) circle (1.5pt) node[above] {$p$};
\fill (0.10,0.58) circle (1.5pt) node[right] {$q$};
\draw[red, thick] (-0.18,0.98) -- (0.10,0.58);
\node[red!70!black, right] at (0.22,1.22) {$|p-q|$};
\draw[orange!85!black, very thick] plot[smooth] coordinates {(-0.18,0.98) (-0.95,1.18) (-1.45,0.42) (-1.22,-0.62) (-0.05,-1.25) (0.85,-0.92) (1.20,0.25) (0.72,0.88) (0.10,0.58)};
\node[orange!85!black, below] at (0,-1.7) {$d_\gamma(p,q) \gg |p-q|$};
\end{scope}
\begin{scope}[xshift=7.2cm]
\node at (0,2.15) {trunk / supertrunk};
\draw[gray!60] (-1.9,1.25) -- (1.9,1.25);
\draw[gray!60] (-1.9,0) -- (1.9,0);
\draw[gray!60] (-1.9,-1.25) -- (1.9,-1.25);
\draw[thick, blue] plot[smooth] coordinates {(-1.35,1.70) (-0.80,0.90) (-0.15,1.38) (0.62,0.52) (1.25,1.68)};
\draw[thick, blue] plot[smooth] coordinates {(-1.45,0.55) (-0.72,-0.48) (-0.15,0.30) (0.52,-0.82) (1.32,0.38)};
\draw[thick, blue] plot[smooth] coordinates {(-1.18,-0.78) (-0.55,-1.70) (0.15,-0.92) (0.86,-1.58) (1.42,-0.72)};
\foreach \x/\y in {-0.52/1.25,0.98/1.25,-1.02/0,-0.36/0,0.26/0,1.04/0,-0.86/-1.25,0.47/-1.25}{\fill (\x,\y) circle (1.2pt);} 
\node at (0,-2.02) {$\max_t \lvert \gamma \cap h_v^{-1}(t) \rvert$};
\end{scope}
\end{tikzpicture}
\caption{Left: nearby points with long arclength separation contribute to distortion. Right: intersections with level planes measure trunk in a chosen direction; taking the worst direction leads to a supertrunk-type quantity.}
\label{fig:distortion-trunk-schematic}
\end{figure}
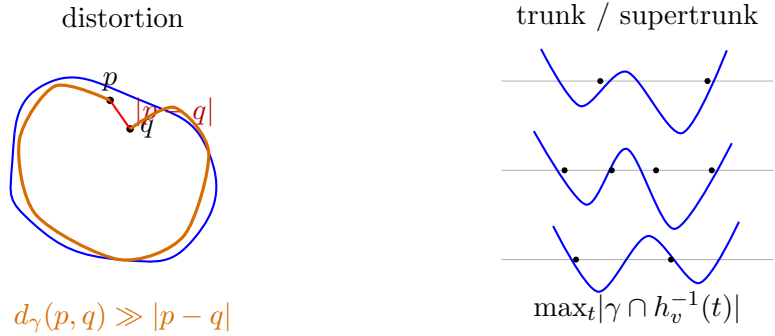

\subsection{Distortion}

For an embedded curve $\gamma$, its distortion is
\[
 \deltaK(\gamma)=
 \sup_{p,q\in\gamma}
 \frac{d_\gamma(p,q)}{|p-q|},
\]
where $d_\gamma(p,q)$ denotes the shorter arclength distance between $p$ and $q$ along $\gamma$.  The distortion of a knot type is
\[
 \deltaK(K)=\inf_{\gamma\in K}\deltaK(\gamma).
\]
Distortion was studied by Gromov and by Denne--Sullivan, who proved a nontrivial lower bound for nontrivial tame knots; Pardon later established strong lower bounds for certain torus knots \cite{DenneSullivan,Pardon}.

\begin{proposition}[Density gives a distortion lower bound]\label{prop:distortion-density}
For every embedded closed curve \(\gamma\),
\[
 \deltaK(\gamma)
 \geq
 \frac{1}{2}\frac{\len(\gamma)}{\diam(\gamma)}.
\]
Consequently,
\[
 \deltaK(\gamma)
 \geq
 \frac14\frac{\len(\gamma)}{R_{\min}(\gamma)}.
\]
\end{proposition}

\begin{proof}
Choose a point \(p\in\gamma\) and let \(q\) be a point reached after traveling arclength \(\len(\gamma)/2\) from \(p\) along one orientation of the closed curve.  The two arcs of \(\gamma\) connecting \(p\) and \(q\) have equal length \(\len(\gamma)/2\).  Therefore the shorter arclength distance satisfies
\[
 d_\gamma(p,q)=\len(\gamma)/2.
\]
Since \(|p-q|\leq\diam(\gamma)\), the definition of distortion gives
\[
 \deltaK(\gamma)
 \geq
 \frac{\len(\gamma)/2}{\diam(\gamma)}.
\]
Finally, \(\diam(\gamma)\leq2R_{\min}(\gamma)\) gives the second inequality.
\end{proof}

For \(D=R_{\min}\) this can be written at the representative level as
\[
 \deltaK(\gamma)
 \geq
 \frac14\rho_R(\gamma)
 =
 \frac14\frac{\Rop(\gamma)}{\CRad_R(\gamma)}.
\]
This is a statement about a fixed representative \(\gamma\), not directly about optimized knot invariants.  Taking infima over representatives gives
\[
 \deltaK(K)\geq \frac14\rho_R(K),
\]
if the same class of representatives is used.  No equality with \(\Rop(K)/\CRad_R(K)\) is implied, since the minimizing representatives for ropelength and compression may differ.
Thus high density inside a fixed enclosing scale forces large distortion.  This implication is one-directional: large enclosing radius does not imply small distortion, because distortion is sensitive to pairs of points that are close in space but far along the curve.

\subsection{Trunk and supertrunk}

Let \(h_v(x)=\langle x,v\rangle\) be a height function in the direction \(v\in S^2\).  For a representative \(\gamma\), set
\[
 \trunk_v(\gamma)=\max_{t\in\RR}|\gamma\cap h_v^{-1}(t)|,
\]
where the maximum is taken over regular levels.  The trunk of \(\gamma\) is the minimum of \(\trunk_v(\gamma)\) over directions or, equivalently in the usual formulation, over Morse height functions; the trunk of a knot type is then obtained by minimizing over representatives.  Following the viewpoint used in work of Blair--Ozawa, one may define a supertrunk-type quantity by taking the maximum over directions before minimizing over representatives:
\[
 \strunk(\gamma)=\max_{v\in S^2}\trunk_v(\gamma),
 \qquad
 \strunk(K)=\inf_{\gamma\in\Hg(K)}\strunk(\gamma).
\]
This is analogous in spirit to Kuiper's superbridge index, where one takes the worst direction before optimizing the embedding; see \cite{Kuiper,BlairOzawa}.

The relation between compression and trunk is not an equality, but the geometry suggests natural bounds.  Suppose $\Thi(\gamma)\geq r$ and $\gamma$ is contained in a ball of radius $R$.  If a level plane intersects $\gamma$ in $N$ points, then the tube around $\gamma$ gives, heuristically, $N$ disjoint disks of radius comparable to $r$ inside a planar disk of radius comparable to $R$.  This suggests an estimate of the form
\[
 N\lesssim C\left(\frac{R}{r}\right)^2.
\]
At the level of scale-free compression radius, this becomes
\[
 \strunk(\gamma)
 \lesssim
 C\CRad_R(\gamma)^2,
\]
or equivalently
\[
 \CRad_R(\gamma)\gtrsim c\sqrt{\strunk(\gamma)}.
\]
There are two caveats.  First, containment of the curve in a ball does not, by itself, identify an optimal planar cross-section for every direction.  Second, the intersection points with a level plane do not automatically determine disjoint disks of radius exactly $r$ in that plane; this requires control of the angle of intersection and the local geometry of the tube.  Thus the preceding paragraph should be read as motivation, not as a proof.  A rigorous version would likely require a transversality assumption or an averaged coarea-type estimate for intersections of the tube with level planes.

This heuristic motivates the following conjectural direction.

\begin{conjecture}[Compression lower bound from supertrunk]
There is a universal constant $c>0$ such that for every knot type $K$,
\[
 \CRad_R(K)
 \geq
 c\sqrt{\strunk(K)}.
\]
Here $\strunk(K)$ denotes an appropriate supertrunk-type invariant.
\end{conjecture}

\begin{question}[Density and trunk]
Is there a universal function $f$ such that
\[
 \rho_R(K)
 \geq
 f(\trunk(K))?
\]
More generally, can one obtain lower bounds for $\rho_D(K)$ or $\CRad_D(K)$ in terms of trunk, supertrunk, superbridge number, or representativity?
\end{question}

The slogan is that trunk is a sectional shadow of density, while supertrunk is a directional shadow of compression.  Making this precise appears to require a careful analysis of tube intersections with level planes.

%==================================================
\section{Questions and further directions}\label{sec:questions}
%==================================================

We collect several questions suggested by the preceding framework.

\begin{question}[Simultaneous minimizers]
For which knot types $K$ and which size functionals $D$ does there exist a representative $\gamma\in\Hg(K)$ such that
\[
 \gamma\in\mathcal I_{\Rop}(K)
 \cap
 \mathcal I_{\rho_D}(K)
 \cap
 \mathcal I_{\CRad_D}(K)?
\]
Equivalently, when does
\[
 \Rop(K)=\rho_D(K)\CRad_D(K)
\]
hold?
\end{question}

\begin{question}[Comparison of ideal strata]
How different can the ropelength ideal stratum, the density ideal stratum, and the compression ideal stratum be?  Can one give examples in which these strata are disjoint?
\end{question}

\begin{question}[Natural strictness and the trefoil]
For a standard nontrivial knot such as the trefoil and for natural choices $D=\diam$ or $D=R_{\min}$, is the inequality
\[
 \Rop(K)\geq \rho_D(K)\CRad_D(K)
\]
strict?  Equivalently, do the ropelength ideal stratum, the density ideal stratum, and the compression ideal stratum fail to have a common minimizing sequence?  This question is deliberately separated from Example~\ref{ex:artificial-strictness}, which proves strictness only for an artificial size functional.
\end{question}

\begin{question}[Connected sum]
How do $\rho_D(K)$ and $\CRad_D(K)$ behave under connected sum?  Is either invariant subadditive, superadditive, or nonmonotone under connected sum for natural choices of $D$?
\end{question}

\begin{question}[Representativity and bridge-type complexity]
Do large representativity, large trunk, or large superbridge number force lower bounds for compression radii or density?  This question is kept within geometric knot invariants closely related to level surfaces and projections.
\end{question}

\begin{question}[Choice of size functional]
Which size functional best separates density from compression?  In particular, how different are the invariants obtained from $\diam$, $R_{\min}$, $R_g$, pairwise $L^p$-spread, and convex-hull size?
\end{question}

\begin{question}[Computability]
For fixed $n$, can $\rho_{D,n}(K)$ and $\CRad_{D,n}(K)$ be effectively approximated by finite-dimensional optimization?  Can these quantities be used as computable surrogates for the continuous invariants?
\end{question}

\begin{remark}[Discrete surrogates]
A concrete instance of the surrogate strategy is carried out in
\cite{OzawaDiscrete}, where the ropelength filtration is replaced by
lattice-filtered move graphs whose components and merge scales are computed by
finite graph search.  Discrete density and compression-radius profiles are
developed in \cite{OzawaDiscretePDensity}, where lattice length is normalized
by chord-length spread functionals and compression radii are defined using a
raw lattice-thickness convention.  The experiments in \cite{OzawaDiscrete}
produce verified move-path certificates; for example, the seed-generated BFACF
merge scales of supplied mirror seed pairs are $32$ for the figure-eight knot
and $44$ for the amphichiral knot $6_3$.  The latter is accompanied by a
verified move-path certificate and was independently checked by a second
implementation of the BFACF transition rules.  The supplementary code and data
are archived on Zenodo \cite{OzawaDiscreteData}.  Together, these works provide
computable finite-state surrogates for density, compression, and filtration
phenomena, while the relation between these lattice profiles and the continuous
invariants remains open.
\end{remark}

\begin{question}[Filtered densities and compression radii]
Following \cite{OzawaIdealStratum,OzawaDiscrete,OzawaDiscretePDensity}, filter $\Hg(K)$ by the
sublevel sets $\{\gamma\mid \CRad_D(\gamma)\leq\Lambda\}$ or
$\{\gamma\mid \rho_D(\gamma)\leq\Lambda\}$, and study the admissible
components and merge scales of the resulting spaces.  How do these persistence
objects depend on the choice of $D$, and how do they compare with the
ropelength filtration?  In particular, do amphichiral knot types exhibit
positive mirror merge barriers in these filtrations, as they do experimentally
in the BFACF lattice model, where the supplied $4_1$ and $6_3$ mirror seed
pairs have barriers $2$ and $4$ respectively \cite{OzawaDiscrete}?
\end{question}

%==================================================
\section*{Acknowledgments}
%==================================================

\paragraph{Use of generative AI.}
The author used ChatGPT (OpenAI) as an interactive aid during the preparation of this manuscript, including for mathematical discussion, the exploration of possible formulations and proof strategies, and the improvement of exposition. All definitions, statements, proofs, computations, and references were independently checked and verified by the author, who takes full responsibility for the accuracy, originality, and content of the manuscript.

%==================================================

\end{document}